\documentclass{article}

\usepackage{amsmath,amsfonts,amsthm,latexsym}

\newcommand{\N}{\mathbb{N}}

\newcommand{\R}{\mathbb{R}}
\newcommand{\rar}{\mbox{$\rightarrow$}}

\newcommand{\Si}{\Sigma}

\newcommand{\T}{\mathbb{T}}
\newcommand{\tr}{\Delta}

\newcommand{\Z}{\mathbb{Z}}

\newtheorem{theorem}{Theorem}[section]
\newtheorem{corollary}[theorem]{Corollary}
\newtheorem{proposition}[theorem]{Proposition}

\theoremstyle{definition}
\newtheorem{ex}[theorem]{Example}
\newtheorem{definition}[theorem]{Definition}

\theoremstyle{remark}
\newtheorem{rem}[theorem]{Remark}

\begin{document}

\title{Linear positive control systems on time scales. Controllability.\thanks{Supported by the Bialystok University of Technology under grant No. S/WI/2/2011}}

\author{Zbigniew Bartosiewicz\thanks{Bia{\l}ystok University of Technology,  Faculty of Computer Science, Wiejska 45A, 15-351~Bialystok, Poland (\tt z.bartosiewicz@pb.edu.pl)}}

\maketitle

\begin{abstract}                          
Positive linear systems on arbitrary time scales are studied. The theory developed in the paper unifies and extends concepts and results known for continuous-time and discrete-time systems. A necessary and sufficient condition for a linear system on a time scale to be positive is presented. Properties of positive reachable sets are investigated and characterizations of various controllability properties are presented. A modified Gram matrix of the system is used to state necessary and sufficient condition of positive reachability of a positive system on an arbitrary time scale.
\end{abstract}

\section{Introduction}
In many applications the variables that appear in a mathematical description take only positive or
nonnegative values. Examples of such systems can be found in \cite{FR} and \cite{Ka}, where also a theory of
linear positive systems was developed. Usually the systems that are studied fall into two separate classes:
continuous-time systems and discrete-time systems. In \cite{Ka} all the problems are studied twice in these two
separate settings. The characterizations of properties of positive systems for these two classes are sometimes
similar, or even identical, and sometimes essentially distinct.

Stefan Hilger in his Ph.D. thesis \cite{H} started the most successful attempt to unify the theories of
continuous-time systems and discrete-time systems into one theory. It is based on the concept of time scale and
the calculus on time scales. A time scale is a model of time. Time may be discrete or continuous, or partly
continuous and partly discrete. The concepts of standard derivative used in the case of continuous time and
forward difference used in the discrete time are unified into one concept of delta derivative. This allows to
consider delta differential equations on arbitrary time scales. They generalize standard differential equations
and difference equations. Theory of dynamical systems on time scales was developed in \cite{Bh}. Special
attention was paid to linear delta differential equations.

The interest in control systems on time scales dates back to 2004. The first results have concerned
controllability, observability and realizations of linear constant-coefficient and varying-coefficient control
systems with outputs \cite{BP2}, \cite{BP3}. Since then the literature on control systems on time scales has been rapidly growing, including also nonlinear systems.

Controllability of continuous-time and discrete-time linear positive systems has been a subject of research since late 1980's \cite{CS}, \cite{FMT}, \cite{OMK}, \cite{RJ}. Discrete-time systems appeared to be easier to deal with and it seems that positive cotrollability of such systems is now fully understood (see e.g. \cite{BRS}, \cite{C}, \cite{Ka1}). On the other hand, only recently it was discovered that positive reachability of continuous-time systems requires very restrictive conditions to be satisfied \cite{CA}, \cite{V}. Thus criteria of positive reachability for discrete-time systems and continuous-time systems are essentially different.

In this paper we study linear positive constant-coefficients systems on arbitrary time scales. The results
presented here unify and extend corresponding results obtained for linear positive continuous-time and
discrete-time systems. We prove necessary and sufficient conditions for a linear system $x^\Delta=Ax+Bu$ on a
time scale $\T$ to be positive. They involve the matrices $A$ and $B$ and the graininess function of the time
scale, which describes distribution of the instances of time. We also study two controllability properties of
positive linear systems: positive accessibility and positive reachability. Accessibility appears to be a property whose characterization does not depend on time scale. It is equivalent to standard controllability and expressed with the aid of the Kalman controllability matrix. As the criteria for positive reachability are completely different for continuous-time and discrete-time systems we have tried to develop methods which would result in the same statements for different time scales. We introduce a modified Gram matrix for a system on time scales, for which we select columns of $B$ and choose different sets of integration for different columns. We prove that the system is positively reachable on an interval if and only if such a Gram matrix is monomial, i.e. each its column and each its row contain exactly one positive element. Then we show that from this characterization we can deduce known criteria for positive reachability of continuous-time and discrete-time systems. We also show that on nonhomogeneous time scales many properties known to hold for homogeneous time scales are no longer true.

In Section~\ref{sec2} we recall basic material on positive systems, time scales and linear systems on time scales. Section~\ref{sec3} is devoted to positive control systems and Section~\ref{sec4} to positive accessibility and positive reachability.

\section{Preliminaries}\label{sec2}
We introduce here the main concepts, recall definitions and facts, and set notation. For more information on
positive continuous-time and discrete-time systems, the reader is referred to e.g. \cite{FR}, and for
information on time scales calculus, to e.g. \cite{Bh}.

\subsection{Positive math}
By $\mathbb{R}$ we shall denote the set of all real numbers, by $\mathbb{Z}$ the set of integers, and by
$\mathbb{N}$ the set of natural numbers (without $0$). We shall also need the set of nonnegative real numbers,
denoted by $\R_+$ and the set of nonnegative integers $\Z_+$, i.e. $\N\cup\{0\}$. Similarly, $\R^k_+$ will mean
the set of all column vectors in $\R^k$ with nonnegative components and $\R^{k\times p}_+$ will consist of
$k\times p$ real matrices with nonnegative elements. If $A\in \R^{k\times p}_+$ we write $A\geq 0$ and say that
$A$ is \emph{nonnegative}. A nonnegative matrix $A$ will be called \emph{positive} if at least one of its
elements is greater than $0$. Then we shall write $A>0$.

A positive column or row vector is called \emph{monomial} if one of its components is positive and all the other are zero. A monomial column in $\mathbb{R}^n_+$ has the form $\alpha e_k$ for some $\alpha>0$ and $1\leq k\leq n$, where $e_k$ denotes the column with 1 at the $k$th position and other elements equal 0. Then we say that the column is \emph{$k$-monomial}. An $n\times n$ matrix $A$ is called \emph{monomial} if  all columns and rows of $A$ are monomial. Then $A$ is invertible and its inverse is also positive. Moreover, we have the following important fact.

\begin{proposition}
A positive matrix $A$ has a positive inverse if and only if $A$ is monomial.
\end{proposition}

It will be convenient to extend the set of all real numbers adding one element. It will be denoted by $\infty$ and
will mean the positive infinity. We set $\bar{\R}:=\R\cup \{\infty\}$ and $\bar{\R}_+:=\R_+\cup \{\infty\}$. If
$a\in \R$ then we define $a+\infty=\infty$. Moreover, for $a\in\R$ and $a>0$ we set $a/0=\infty$ and
$a/\infty=0$. Of course $\infty>0$. If a matrix $A$ has elements from $\bar{\R}$, then the notions of
nonnegativity and positivity have the same meanings as before and are denoted in the same way. Addition of such
matrices is defined in the standard way, but we shall not need multiply or invert such matrices.

\subsection{Calculus on time scales}

Calculus on time scales is a generalization of the standard differential calculus and the calculus of finite
differences.

 A {\it time scale} $\T$ is an arbitrary nonempty closed subset of the set $\R$
of real numbers. In particular $\T=\R$, $\T=h\Z$ for $h>0$ and $\T=q^{\N}:=\{ q^k, k\in\N\}$ for $q>1$ are time scales. We assume that $\T$ is a
topological space with the relative topology induced from $\R$. If $t_0,t_1\in\T$, then $[t_0,t_1]_\T$ denotes the
intersection of the ordinary closed interval with $\T$. Similar notation is used for open, half-open or infinite
intervals.

For $t \in \T$ we define
 the {\it forward jump operator} $\sigma:\T \rar \T$ by
$\sigma(t):=\inf\{s \in \T:s>t\}$ if $t\neq\sup\T$ and $\sigma(\sup\T)=\sup\T$ when $\sup\T$ is finite;
the {\it backward jump operator} $\rho:\T \rar \T$ by
$\rho(t):=\sup\{s \in \T:s<t\}$ if $t\neq\inf\T$ and $\rho(\inf\T)=\inf\T$ when $\inf\T$ is finite;
  the {\it forward graininess function} $\mu:\T \rar [0,\infty)$ by
$\mu(t):=\sigma(t)-t$;
 the {\it backward graininess function} $\nu:\T \rar [0,\infty)$ by
$\nu(t):=t-\rho(t)$.

If $\sigma(t)>t$, then $t$ is called {\it right-scattered}, while if $\rho(t)<t$, it is called {\it
left-scattered}. If $t<\sup\T$ and $\sigma(t)=t$ then $t$ is called {\it right-dense}. If $t>\inf\T$ and
$\rho(t)=t$, then $t$ is {\it left-dense}.

The time scale $\T$ is \emph{homogeneous}, if $\mu$ and $\nu$ are constant functions. When $\mu\equiv 0$ and $\nu\equiv 0$, then $\T=\R$ or $\T$ is a closed interval (in particular a half-line). When $\mu$ is constant and greater than $0$, then $\T=\mu\Z$.

Let $\T^k:=\{t \in \T: t\;\text{ is nonmaximal or left-dense}\}$. Thus $\T^k$ is got from $\T$
by removing its maximal point if this point exists and is left-scattered.

Let $f:\T \rar \R$ and $t \in \T^k$.  The \emph{delta derivative of $f$ at $t$}, denoted by $f^{\tr}(t)$, is the
real number with the property that given any $\varepsilon$ there is a neighborhood
$U=(t-\delta,t+\delta) \cap \T$  such that
\[|(f(\sigma(t))-f(s))-f^{\tr}(t)(\sigma(t)-s)| \leq \varepsilon|\sigma(t)-s|\]
for all $s \in U$. If $f^{\tr}(t)$ exists, then we say that \emph{$f$ is delta differentiable at $t$}. Moreover, we say that $f$ is {\it delta differentiable} on $\T^k$ provided $f^{\tr}(t)$
exists for all $t\in \T^k$.

\begin{ex}
If $\T=\R$, then  $f^{\tr}(t)=f'(t)$.
 If $\T=h\Z$, then
$f^{\tr}(t)=\frac{f(t+h)-f(t)}{h}$.
If $\T=q^{\N}$, then $f^{\tr}(t)=\frac{f(qt)-f(t)}{(q-1)t}$.
\end{ex}

 A function
$f:\T \rar \R$ is called {\it rd-continuous} provided it is continuous at right-dense points in $\T$ and its
left-sided limits exist (finite) at left-dense points in $\T$.
If $f$ is continuous, then it is rd-continuous.

A function $F:\T \rar \R$ is called an {\it antiderivative} of $f: \T \rar \R$
provided $F^{\tr}(t)=f(t)$ holds for all $t \in \T^k$. Let $a,b\in\T$. Then the \emph{delta integral} of $f$ on the interval $[a,b)_\T$ is defined by
\[ \int_a^b f(\tau) \tr \tau :=\int_{[a,b)_\T} f(\tau) \tr \tau := F(b)-F(a). \]

It is more convenient to consider the half-open interval $[a,b)_\T$ than the closed interval $[a,b]_\T$ in the definition of the integral. If $b$ is a left-dense point, then the value of $f$ at $b$ would not affect the integral. On the other hand, if $b$ is left-scattered, the value of $f$ at $b$ is not essential for the integral (see Example~\ref{e10}). This is caused by the fact that we use delta integral, corresponding to the forward jump function.

We have a natural property:
\[ \int_a^b f(\tau)\Delta\tau=\int_a^c f(\tau)\Delta\tau + \int_c^b f(\tau)\Delta\tau \]
for any $c\in (a,b)_T$.
Moreover, if $f$ is rd-continuous, $f(t) \geq 0$ for all $a
\leq t < b$ and $\int\limits_a^b f(\tau) \tr \tau =0 $, then $f\equiv 0$.

\begin{ex} \label{e10}
If $\T=\R$, then $\int\limits_a^b f(\tau) \tr \tau=\int\limits_a^b f(\tau)d\tau$, where the
integral on the right is the usual Riemann integral. If $\T=h\Z$, $h>0$, then $\int\limits_a^b
f(\tau)\tr\tau=\sum\limits_{t=\frac{a}{h}}^{\frac{b}{h}-1}f(th)h$ for $a<b$.
\end{ex}

\subsection{Linear systems on time scale}

Let us consider the system of delta differential equations on a time scale $\T$:
\begin{equation} \label{sys}
 x^{\tr}(t)=Ax(t),
 \end{equation}
where $x(t)\in\R^n$ and $A$ is a constant $n\times n$ matrix.

\begin{rem} If $\T=\R$, then \eqref{sys} is a system of ordinary differential equations $x'=Ax$. But for $\T=\Z$, \eqref{sys} takes the difference form  $x(t+1)-x(t)=Ax(t)$, which can be transformed to the shift form $x(t+1)=(I+A)x(t)$. Thus to compare the definitions and the results stated for delta differential systems in the case $\T=\Z$ with those that were obtained for discrete-time systems in the shift form, one has to take this into account. One can easily transform the difference form to the shift form and vice versa.
\end{rem}

\begin{proposition}  \label{pr1}
Equation \eqref{sys} with initial condition $x(t_0)=x_0$ has a unique forward solution defined for all $t \in
[t_0,+\infty)_\T$.
\end{proposition}

The \emph{matrix exponential function} (at $t_0$) for $A$ is defined as the unique forward solution of the
matrix differential equation $X^{\tr}=AX$, with the initial condition $X(t_0)=I$. Its value at $t$ is denoted by
$e_A(t,t_0)$.

\begin{ex}
 If $\T=\R$, then $e_A(t,t_0)=e^{A(t-t_0)}$.
 If $\T=h\Z$, then $e_A(t,t_0)=(I+A)^{(t-t_0)/h}$.
If $\T=q^{\N}$, $q>1$, then
 $e_A(q^kt_0,{t_0})=\prod_{i=0}^{k-1}(I+(q-1)q^it_0A)$ for $k\geq 1$ and $t_0\in\T$.
\end{ex}

\begin{proposition} \label{p10}
The following properties hold for every $t,s,r\in\T$ such that $r \leq s \leq t$:\\
{\bf i)} $e_A(t,t) = I$; \\
{\bf ii)} $e_A(t,s)e_A(s,r)=e_A(t,r)$;\\
\end{proposition}

Let us consider now a nonhomogeneous system
\begin{equation} \label{sys1}
x^{\tr}(t)=Ax(t)+f(t)
\end{equation}
where $f$ is rd-continuous.

\begin{theorem} \label{nonh}
Let $t_0\in\T$. System~\eqref{sys1} for the initial condition $x(t_0)=x_0$ has a unique forward solution of the
form
 \begin{equation} \label{vc}
  x(t)=e_A(t,t_0)x_0 + \int_{t_0}^t
  e_A(t,\sigma(\tau))f(\tau)\tr\tau.
 \end{equation}
\end{theorem}

\section{Positive control systems} \label{sec3}

Let $n\in\N$ be fixed. From now on we shall assume that the time scale $\T$ consists of at least $n+1$
elements.

Let us consider a linear control system, denoted by $\Si$, and defined on the time scale $\T$:
\begin{equation}\label{ukl1}
   x^{\tr}(t)=Ax(t)+Bu(t)
           \end{equation}
where $t\in\T$, $x(t)\in\R^n$ and $u(t)\in\R^m$.

We assume that the control $u$ is a piecewise continuous function defined on some interval $[t_0,t_1)_\T$, depending on $u$, where $t_0\in\T$ and $t_1\in \T$ or $t_1=\infty$.  We shall assume that at each point $t\in [t_0,t_1)_\T$, at which $u$ is not continuous, $u$ is right-continuous and has a finite left-sided limit if $t$ is left-dense. This allows to solve \eqref{ukl1} step by step. Moreover, for a finite $t_1$ we can always evaluate $x(t_1)$. For $t_1$ being left-scattered we do not need the value of $u$ at $t_1$, and for a left-dense $t_1$ we just take a limit of $x(t)$ at $t_1$.

\begin{definition}
We say that system $\Sigma$ is \emph{positive} if for any $t_0\in\T$, any initial condition $x_0\in\R^n_+$, any
control $u : [t_0,t_1)_\T \rar \R^m_+$ and any $t\in [t_0,t_1]_\T$, the solution $x$ of \eqref{ukl1} satisfies
$x(t)\in\R^n_+$.
\end{definition}

By the separation principle we have the following characterization.

\begin{proposition}
The system $\Sigma$ is positive if and only if $e_A(t,t_0)\in\R^{n\times n}_+$ for every $t,t_0\in\T$ such that
$t\geq t_0$, and $B\in\R^{n\times m}_+$.
\end{proposition}

The proof is very similar to the proof of the continuous-time case.

To state criteria of nonnegativity of the exponential matrix, let $\bar{\mu}=\sup \{ \mu(t) : t\in\T \}$ and
$A_\T:= A+I/\bar{\mu}$, where $I/\infty$ means the zero $n\times n$ matrix and $I/0$ is a diagonal matrix with
$\infty$ on the diagonal. Thus for $\T=\R$, $A_\T$ is obtained from $A$ by replacing the elements on the
diagonal by $\infty$, for $\T=\Z$, $A_\T=A+I$, and for $\T=q^\N$, $A_\T=A$.

\begin{theorem}
The exponential matrix $e_A(t,t_0)$ is nonnegative for every $t,t_0\in\T$ such that $t\geq t_0$ if and only if
$A_\T\in\bar{\R}^{n\times n}_+$.
\end{theorem}

\begin{proof}
``$\Leftarrow$''  Assume that $A_\T\geq 0$. If $\mu(t_0)>0$, then $A+I/\mu(t_0)\geq 0$. This means that
$e_A(\sigma(t_0),t_0)=\mu(t_0)A+I \geq 0$. If $\mu(t_0)=0$, then for $t>t_0$ and close to $t_0$, $I+A(t-t_0)>0$.
The last term approximates $e_A(t,t_0)$. Since the exponential matrix is continuous (with respect to $t$), then
also $e_A(t,t_0)>0$ for $t>t_0$ and close to $t_0$. To achieve nonnegativity of $e_A(t,t_0)$ for all $t\in\T$,
$t>t_0$, we have to use the semigroup property of the exponential matrix: $e_A(t,s)e_A(s,\tau)=e_A(t,\tau)$ for
$\tau<s<t$ and $\tau,s,t\in\T$. \\
``$\Rightarrow$'' Assume that $e_A(t,t_0)$ is nonnegative for $t,t_0\in\T$ such that $t\geq t_0$. Suppose first
that $\bar{\mu}>0$ and chose $t_0\in\T$ with $\mu(t_0)>0$. Then $e_A(\sigma(t_0),t_0)=I+\mu(t_0)A\geq 0$. This
means that also $A+I/\mu(t_0)\geq 0$. As it holds for all $t_0\in\T$ with $\mu(t_0)>0$,
$A_\T=A+I/\bar{\mu}$ is nonnegative. If $\bar{\mu}=0$, then $\T$ is a standard interval. The exponential matrix
is then standard $e^{A(t-t_0)}$. For $t$ close to $t_0$, it may be approximated by $I+A(t-t_0)$. Nonnegativity
of the exponential matrix implies that $I+A(t-t_0)>0$ for $t>t_0$ and close to $t_0$. This holds only if all
elements of $A$ outside the diagonal are nonnegative. Thus again $A_\T$ is nonnegative.
\end{proof}

\begin{corollary}
The system $\Sigma$ is positive if and only if $A_\T\in\bar{\R}^{n\times n}_+$ and $B\in\R^{n\times m}_+$.
\end{corollary}

\begin{rem}
\textup{An $n\times n$ matrix with nonnegative elements outside the diagonal is called a \emph{Metzler matrix}. Thus in
the continuous-time case, the exponential matrix $e_A(t,t_0)$ is nonnegative for every  $t>t_0$ if and only if $A$ is a
Metzler matrix. In that case the elements on the diagonal may be arbitrary. On the other hand, if the time scale
$\T$ is the set $\Z$ of integer numbers, then $\mu\equiv 1$ and nonnegativity of the exponential matrix is
equivalent to $A+I\geq 0$. In that case the delta differential equation $x^\Delta(k)=Ax(k)$
may be rewritten in the shift form as $x(k+1)=(A+I)x(k)$.
Thus the condition $A+I\geq 0$ agrees with the necessary and sufficient condition of nonnegativity for
discrete-time systems of the form $x(k+1)=Fx(k)$, where $k\in\Z$ (see \cite{FR,Ka}).}
\end{rem}

\section{Controllability} \label{sec4}

If $\Sigma$ is a positive system, then for a nonnegative initial condition $x_0$ and a nonnegative control $u$,
the trajectory $x$ stays in $\R^n_+$. One may be interested in properties of the reachable sets of the system.
For simplicity we assume that the initial condition is $x_0=0$. Let $x(t_1,t_0,0,u)$ mean the trajectory of the system
corresponding to the initial condition $x(t_0)=0$ and the control $u$, and evaluated at time $t_1$. We shall define various controllability properties.

\begin{definition} Let $t_0,t_1\in\T$, $t_0<t_1$.
The \emph{positive reachable set} (from $0$) of the system $\Sigma$ on the interval $[t_0,t_1]_\T$ is the set
$\mathcal{R}_+^{[t_0,t_1]}$ consisting of all $x(t_1,t_0,0,u)$,
where $u$ is a nonnegative control on $[t_0,t_1)_\T$. \\
   The
\emph{positive reachable set} (from $0$) \emph{for the initial time} $t_0$ of $\Sigma$ is
\[  \mathcal{R}_+^{t_0}=\bigcup_{t_1\in\T, t_1>t_0} \mathcal{R}_+^{[t_0,t_1]} \]
and the \emph{positive reachable set} (from $0$) of $\Sigma$ is
\[  \mathcal{R}_+=\bigcup_{t_0\in\T} \mathcal{R}_+^{t_0}. \]
The system $\Sigma$ is \emph{positively accessible on $[t_0,t_1]_\T$} if $\mathcal{R}_+^{[t_0,t_1]}$ has a nonempty interior, $\Sigma$ is \emph{positively accessible for the initial time $t_0$} if  $\mathcal{R}_+^{t_0}$ has a nonempty interior and $\Sigma$ is \emph{positively accessible} if  $\mathcal{R}_+$ has a nonempty interior.\\
The system $\Sigma$ is \emph{positively reachable on $[t_0,t_1]_\T$} if $\mathcal{R}_+^{[t_0,t_1]}=\R^n_+$,
$\Sigma$ is \emph{positively reachable for the initial time $t_0$} if $\mathcal{R}_+^{t_0}=\R^n_+$ and $\Sigma$ is
\emph{positively reachable} if $\mathcal{R}_+=\R^n_+$.
\end{definition}

\begin{rem}
\textup{Accessibility was first introduced for nonlinear systems for which it is a good substitute of reachability, as
the latter is often too restrictive property. The same happens for positive systems. Positive accessibility means
precisely accessibility but with nonnegative controls.}
\end{rem}

The following implications follow directly from definitions:
\begin{proposition}
$\Sigma$ is positively accessible on $[t_0,t_1]_\T$ $\Rightarrow$ $\Sigma$ is positively accessible for the initial time $t_0$ $\Rightarrow$ $\Sigma$ is positively accessible.\\
$\Sigma$ is positively reachable on $[t_0,t_1]_\T$ $\Rightarrow$
$\Sigma$ is positively reachable for the initial time $t_0$ $\Rightarrow$ $\Sigma$ is
positively reachable.\\
Positive reachability (on $[t_0,t_1]_\T$) implies positive accessibility (on $[t_0,t_1]_\T$).
\end{proposition}

We have also a useful inclusion:
\begin{proposition}
If $\tau_0<t_0<t_1$ , then $\mathcal{R}_+^{[t_0,t_1]}\subseteq \mathcal{R}_+^{[\tau_0,t_1]}$ and $\mathcal{R}_+^{t_0} \subseteq \mathcal{R}_+^{\tau_0}$.
\end{proposition}
\begin{proof}
Since we start at $x_0=0$, to reach $x_1\in \mathcal{R}_+^{[t_0,t_1]}$ for the initial time $\tau_0$ put $u(\tau)=0$ for $\tau\in [\tau_0,t_0)$. Then switch to the control that was used to reach $x_1$ for the initial time $t_0$.
\end{proof}

\begin{rem}
Since $\T$ may not be homogeneous, in general $\mathcal{R}_+^{t_0}$ depends on $t_0$.
\end{rem}

\begin{theorem} Let $\Sigma$ be a positive system and $t_0,t_1$ be elements of $\mathbb{T}$ such that $[t_0,t_1]_\T$ consists of at least $n+1$ elements. The following conditions are equivalent:\\
a) $\Sigma$ is positively accessible on $[t_0,t_1]_\T$,\\
b) $\Sigma$ is positively accessible for the initial time $t_0$,\\
c) $\Sigma$ is positively accessible,\\
d) ${\rm rank}(B,AB,\ldots,A^{n-1}B)=n$.
\end{theorem}

\begin{proof}
a) $\Rightarrow$ b). This follows from the fact that $\mathcal{R}_+^{[t_0,t_1]}$ is contained in $\mathcal{R}_+^{t_0}$.\\
b) $\Rightarrow$ c). This follows from the fact that $\mathcal{R}_+^{t_0}$ is contained in $\mathcal{R}_+$.\\
c) $\Rightarrow$ d). Assume that $\Sigma$ is positively accessible. The reachable set $\mathcal{R}$ of $\Sigma$ is
the set of states that can be reached with controls that are not necessarily nonnegative. It is clear that
$\mathcal{R}_+$ is contained in $\mathcal{R}$. Moreover, $\mathcal{R}$ is a linear subspace of $\R^n$ (see
\cite{BP3}). Positive accessibility implies that $\mathcal{R}$ contains an open subset. Therefore $\mathcal{R}$
must be equal to $\R^n$, which means that $\Sigma$ is reachable (controllable) from $0$. This is characterized
by the condition ${\rm rank}(B,AB,\ldots,A^{n-1}B)=n$ (see \cite{BP3}).\\
d) $\Rightarrow$ a). The condition ${\rm rank}(B,AB,\ldots,A^{n-1}B)=n$ implies reachability of $\Sigma$ from $0$ on arbitrary interval consisting of at least $n+1$ points
(see \cite{BP3}). Actually the states can be obtained with the aid of piecewise constant controls with at most
$n-1$ switching at fixed instances. Thus the reachable set $\mathcal{R}^{[t_0,t_1]}$ on $[t_0,t_1]_\T$ can be described as the image of a linear map defined on a
finite-dimensional space. This map restricted to positive controls will give the set with the nonempty interior
and this means positive accessibility.
\end{proof}

To study positive reachability let us introduce a modified Gram matrix related to the control system.

\begin{definition} \label{df21}
Let $M\subseteq \{1,\ldots,m\}$ and $t_0,t_1\in\T$, $t_0<t_1$. For each $k\in M$ let $S_k$ be a subset of $[t_0,t_1)_\T$ that is a union of finitely many disjoint intervals of $\T$ of the form $[\tau_0,\tau_1)_\T$, and let $\mathcal{S}_M=\{S_k : k\in M\}$. By \emph{the Gram matrix of system \eqref{ukl1} corresponding to $t_0$, $t_1$, $M$ and $\mathcal{S}_M$} we mean the matrix
\begin{equation}\label{gram}
W:=W_{t_0}^{t_1}(M,\mathcal{S}_M):=\sum_{k\in M} \int_{S_k} e_A(t_1,\sigma(\tau))b_kb_k^Te_A(t_1,\sigma(\tau))^T \Delta\tau.
\end{equation}
\end{definition}

Then we have the following characterization:

\begin{theorem} \label{th21}
Let $t_0,t_1\in\T$, $t_0<t_1$.  System \eqref{ukl1} is positively reachable on $[t_0,t_1]_\T$ iff there are $M\subseteq \{1,\ldots,m\}$ and the family $\mathcal{S}_M=\{S_k : k\in M\}$ of subsets of $[t_0,t_1]_\T$ such that the matrix $W=W_{t_0}^{t_1}(M,\mathcal{S}_M)$ is monomial.
\end{theorem}

\begin{proof}
``$\Leftarrow$''\ Let $\bar{x}\in\R^n_+$. By $\tilde{e}_1$,\ldots, $\tilde{e}_m$ we denote the vectors of the standard basis in $\R^m$. Define control $u:[t_0,t_1)\rar \R^m_+$ by $u(\tau)=\sum_{k\in M} u_k(\tau)\tilde{e}_k$, where $u_k(\tau)=b_k^Te_A(t_1,\sigma(\tau))^TW^{-1}\bar{x}$ for $t\in S_k$ and $u_k(\tau)=0$ for $t\notin S_k$. The control $u$ is nonnegative and
\begin{align*} x(t_1)&=\int_{t_0}^{t_1} e_A(t_1,\sigma(\tau))Bu(\tau)\Delta\tau=
\sum_{k\in M} \int_{t_0}^{t_1} e_A(t_1,\sigma(\tau))b_ku_k(\tau)\Delta\tau\\
&=\sum_{k\in M} \int_{S_k} e_A(t_1,\sigma(\tau))b_kb_k^Te_A(t_1,\sigma(\tau))^TW^{-1}\bar{x}\Delta\tau = \bar{x}.
\end{align*}
Thus \eqref{ukl1} is positively reachable on $[t_0,t_1]_\T$.\\
``$\Rightarrow$''\ Positive reachability implies that all the vectors $e_1,\ldots,e_n$ can be reached using nonnegative controls. Let us fix some $e_i$. Then there is a piecewise continuous nonnegative control $u=(u_1,\ldots,u_m)$ on $[t_0,t_1)_\T$ such that
\[ e_i=\sum_{j=1}^m \int_{t_0}^{t_1} e_A(t_1,\sigma(\tau))b_ju_j(\tau)\Delta\tau. \]
Since all the integrals in the sum are nonnegative, for some $k_i$ the integral
\[ \int_{t_0}^{t_1} e_A(t_1,\sigma(\tau))b_{k_i}u_{k_i}(\tau)\Delta\tau\]
 is an $i$-monomial vector. Then for every $\tau\in [t_0,t_1)_\T$ the vector $e_A(t_1,\sigma(\tau))b_{k_i}u_{k_i}(\tau)$ is either $i$-monomial or $0$. Let $B_i$ be the set all $\tau$ for which $e_A(t_1,\sigma(\tau))b_{k_i}u_{k_i}(\tau)$ is $i$-monomial. Then for $\tau\in B_i$ the matrix \[ e_A(t_1,\sigma(\tau))b_{k_i}b_{k_i}^Te_A(t_1,\sigma(\tau))^T \] is diagonal with the only nonzero element at the $i$th place. The same is true for the matrix $\int_{B_i}e_A(t_1,\sigma(\tau))b_{k_i}b_{k_i}^Te_A(t_1,\sigma(\tau))^T \Delta\tau$. This implies that the matrix
\[ C:=\sum_{i=1}^n \int_{B_i}e_A(t_1,\sigma(\tau))b_{k_i}b_{k_i}^Te_A(t_1,\sigma(\tau))^T \Delta\tau \] is monomial (and diagonal).
Let $M$ consist of all $k_i$ for $i=1,\ldots,n$. Observe that if $k_i=k_j$ for $i\neq j$, then $B_i\cap B_j=\emptyset$. Define $S_k=\bigcup_{k_i=k} B_i$ and let $\mathcal{S}_M=\{S_k : k\in M\}$. Then
\[ C=\sum_{k\in M} \int_{S_k}e_A(t_1,\sigma(\tau))b_{k}b_{k}^Te_A(t_1,\sigma(\tau))^T \Delta\tau=W_{t_0}^{t_1}(M,\mathcal{S}_M),\] so $W_{t_0}^{t_1}(M,\mathcal{S}_M)$ is monomial.
\end{proof}

\begin{corollary} \label{c4}
If the ordinary Gram matrix
\[ W_{t_0}^{t_1}=\int_{t_0}^{t_1} e_A(t_1,\sigma(\tau))BB^Te_A(t_1,\sigma(\tau))^T \Delta\tau \] is monomial, then system \eqref{ukl1} is positively reachable on $[t_0,t_1]_\T$.
\end{corollary}
\begin{proof}
Observe that $W_{t_0}^{t_1}=W_{t_0}^{t_1}(M,\mathcal{S}_M)$ for $M=\{1,\ldots,m\}$ and $S_k=[t_0,t_1)_\T$ for all $k\in M$. Thus positive reachability follows from Theorem~\ref{th21}.
\end{proof}

\begin{rem}
\textup{The condition that $W_{t_0}^{t_1}$ is monomial is not necessary for positive reachability on $[t_0,t_1]_\T$. Consider the system
\begin{equation}\label{eq22}
x^\Delta=\begin{pmatrix} -1 & 1 \\
                          1 & 0
         \end{pmatrix} + \begin{pmatrix} 1 & 1 \\
                                         0 & 1
                           \end{pmatrix}u
\end{equation}
on $\mathbb{T}=\mathbb{Z}$. Choose $t_0=0$ and $t_1=2$. System \eqref{eq22} is positively reachable on $[t_0,t_1]_\T$.
Indeed, let $M=\{1\}$ and $S_1=[0,2)_\T$. Then
\begin{align*} W=&b_1b_1^T + (I+A)b_1b_1^T(I+A)^T\\
=&
\begin{pmatrix} 1  \\
                     0
         \end{pmatrix} (1 , 0)+
         \begin{pmatrix} 0 & 1 \\
                         1 & 1
         \end{pmatrix}\begin{pmatrix} 1  \\
                     0
         \end{pmatrix} (1 , 0)\begin{pmatrix} 0 & 1 \\
                                              1 & 1
                                \end{pmatrix}=\begin{pmatrix} 1 & 0 \\
                                                              0 & 1
                                              \end{pmatrix} \end{align*}
is a monomial matrix. However
\[   W_{t_0}^{t_1}=BB^T + (I+A)BB^T(I+A)^T=\begin{pmatrix} 3 & 3 \\
                                   3 & 6
                           \end{pmatrix} \]
                           is not monomial.}
\end{rem}

\begin{corollary} \label{c5}
If there exists $M\subseteq \{1,\ldots,m\}$ such that the matrix \\${W}_{t_0}^{t_1}(M)=
\int_{t_0}^{t_1} e_A(t_1,\sigma(\tau))\tilde{B}\tilde{B}^Te_A(t_1,\sigma(\tau))^T \Delta\tau$ is monomial, where  $\tilde{B}$ is a submatrix of $B$ consisting of column $b_k, k\in M$, then system \eqref{ukl1} is positively reachable on $[t_0,t_1]_\T$.
\end{corollary}
\begin{proof} Observe that ${W}_{t_0}^{t_1}(M)=W_{t_0}^{t_1}(M,\mathcal{S}_M)$ where $S_k=[t_0,t_1)_\T$ for all $k\in M$. Thus positive reachability follows from Theorem~\ref{th21}.
\end{proof}

\begin{rem} \label{r24}
\textup{The condition that ${W}_{t_0}^{t_1}(M)$ is monomial is not necessary for positive reachability on $[t_0,t_1]_\T$. Let the time scale $\mathbb{T}=\{0\}\cup [1,2]\cup \{3\}$. Consider the system
\begin{equation}\label{eq23}
x^\Delta=\begin{pmatrix} -1 & 0 \\
                          1 & -1
         \end{pmatrix} + \begin{pmatrix} 1  \\
                                         0
                           \end{pmatrix}u.
\end{equation}
The system is positively reachable on $[0,3]_\T$. Indeed, let $M=\{1\}$ and let $S_1=[0,1)_\T\cup [2,3)_\T$. Then
\begin{align*}
  W=\int_{[0,1)_\T} e_A&(3,\sigma(\tau))BB^Te_A(3,\sigma(\tau))^T \Delta\tau\\
   &\mbox{   }+
\int_{[2,3)_\T} e_A(3,\sigma(\tau))BB^Te_A(3,\sigma(\tau))^T \Delta\tau \\
 =
\begin{pmatrix} 0 & 0 \\
                0 & e^{-2}
\end{pmatrix} &
+ \begin{pmatrix} 1 & 0 \\
                                0 & 0
         \end{pmatrix}= \begin{pmatrix} 1 & 0 \\
                                        0 & e^{-2}
         \end{pmatrix}
\end{align*}
is monomial. Observe that we remove here the points $t$ with $\mu(t)=0$. This is essential in order to get a monomial matrix. To calculate the full Gram matrix we have to add to $W$ the following matrix
\[ \int_{[1,2)} e_A(3,\sigma(\tau))BB^Te_A(3,\sigma(\tau))^T d\tau. \]
Its off-diagonal elements are equal to $\int_1^2 (3-\tau)e^{-2(3-\tau)}d\tau$. Since they are positive, ${W}_{t_0}^{t_1}(M)$ is not monomial.}
\end{rem}

For $\T=\R$ we obtain very restrictive conditions for positive reachability.

\begin{proposition}[\cite{CA}] \label{p5}
Let $\T=\R$ and $t_0,t_1\in\R$, $t_0<t_1$. System \eqref{ukl1} is positively reachable on $[t_0,t_1]$ iff $A$ is diagonal and $B$ contains an $n\times n$ monomial submatrix (so $m\geq n$).
\end{proposition}

\begin{proof}
``$\Leftarrow$''\ Let $\tilde{B}$ denote the monomial submatrix of $B$ and let the indices of columns of $\tilde{B}$ form the set $M$. Then $\tilde{B}\tilde{B}^T$ is a diagonal matrix with all the diagonal elements being positive and so is
\[   W_{t_0}^{t_1}(M)=\int_{t_0}^{t_1} e_A(t_1,\sigma(\tau))\tilde{B}\tilde{B}^Te_A(t_1,\sigma(\tau))^T \Delta\tau. \]
 Thus $W_{t_0}^{t_1}(M)$ is monomial, so system \eqref{ukl1} is positively reachable by Corollary~\ref{c5}. Observe that the proof of this implication works for all time scales.

``$\Rightarrow$''\
Assume that the system is positively reachable on $[t_0,t_1]$. From Theorem~\ref{th21} it follows that for some set $M$ and some family $\mathcal{S}_M$ the Gram matrix $W=W_{t_0}^{t_1}(M,\mathcal{S}_M)$ is monomial. Let $j$th column of $W$ be $i$-monomial. Then for some $k\in M$ and for $\tau$ from some subinterval of $[t_0,t_1)$ the $j$th column of the matrix $e^{A(t_1-\tau)}b_kb_k^T(e^{A(t_1-\tau)})^T$ is $i$-monomial. Let $c(\tau)=e^{A(t_1-\tau)}b_k$. Since  the $j$th column of the matrix $c(\tau)c(\tau)^T$ is $i$-monomial, then $c(\tau)$ must be $j$-monomial and eventually $i=j$. This means that at least one column of $e^{A(t_1-\tau)}$ must be $i$-monomial. As the exponential matrix is invertible such a column must be unique. This implies that $b_k$ is monomial. Moreover the $i$-monomial column of $e^{A(t_1-\tau)}$ must be its $i$th column. Otherwise we would get $0$ on the diagonal of the exponential matrix for all $\tau$ from some interval, which is impossible. Thus  $e^{A(t_1-\tau)}$ is diagonal on some interval, which means that $A$ is also diagonal. Now to get all $n$ monomial columns in $W$ we need $n$ different monomial column $b_k$. Thus $B$ contains an $n\times n$ monomial submatrix.
\end{proof}

\begin{rem} \textup{The statement of Proposition~\ref{p5} holds also for $\T=[a,b]$ and for $\T$ being a closed half-line. However it does not hold for disjoint union of closed intervals. The example given in Remark~\ref{r24} may be considered on a bigger time scale: $\T=[-1,0]\cup [1,2]\cup [3,4]$. Neither $A$ nor $B$ satisfy the requirement given in Proposition~\ref{p5}. However the system is positively reachable on $[0,3]_\T$.}
\end{rem}

\begin{corollary}
For $\T=\R$, system \eqref{ukl1} is positively reachable on some $[t_0,t_1]$ iff \eqref{ukl1} is positively reachable on any interval $[\tau_0,\tau_1]$.
\end{corollary}

\begin{corollary}
For $\T=\R$, system \eqref{ukl1} is positively reachable on some $[t_0,t_1]$ iff \eqref{ukl1} is positively reachable. \end{corollary}

For discrete homogeneous time scales the conditions for positive reachability are much less restrictive.

\begin{proposition} \label{p22}
Let $\T=\mu \mathbb{Z}$ for a constant $\mu>0$. Let $t_0\in\T$ and $t_1=t_0+k\mu$ for some $k\in \mathbb{N}$. System \eqref{ukl1} is positively reachable on $[t_0,t_1]_\T$ iff the matrix
  $[B,(I+\mu A)B,\ldots,(I+\mu A)^{k-1}B]$
 contains a monomial submatrix.
\end{proposition}
\begin{proof}
``$\Leftarrow$'' \ Observe that $x(t_1)=\sum_{i=0}^{k-1}\sum_{j=1}^m (I+\mu A)^ib_ju_j(k-1-i)$. If $(I+\mu A)^ib_j=\gamma e_s$ for some $\gamma>0$, then setting $u_j(k-1-i)=1/\gamma$ and all other components and values at different times putting to $0$ we get $x(t_1)=e_s$. This means positive reachability on $[t_0,t_1]_\T$.\\
``$\Rightarrow$''\ By Theorem~\ref{th21} positive reachability implies existence of a set $M$ and subsets $S_k$ of $[t_0,t_1]$ for $k\in M$ such that the matrix
\[ W=\sum_{k\in M} \int_{S_k} e_A(t_1,\sigma(\tau))b_kb_k^Te_A(t_1,\sigma(\tau))^T \Delta\tau \]
is monomial. Moreover
\begin{align*}
 \int_{S_k} e_A(t_1,\sigma(\tau))&b_kb_k^Te_A(t_1,\sigma(\tau))^T \Delta\tau =\\
&\sum_{t\in {S}_k} (I+\mu A)^{(t_1-t-\mu)/\mu}b_kb_k^T((I+\mu A)^{(t_1-t-\mu)/\mu})^T\mu.
\end{align*}
This implies that for every $i=1,\ldots,n$ there are $k\in M$, $t\in S_k$ and $0\leq j\leq n$ such that the $j$th column of $(I+\mu A)^{(t_1-t-\mu)/\mu}b_kb_k^T((I+\mu A)^{(t_1-t-\mu)/\mu})^T$ is $i$-monomial. This means that the column $(I+\mu A)^{(t_1-t-\mu)/\mu}b_k$ is $i$-monomial. But this column is one of the columns of the matrix $[B,(I+\mu A)B,\ldots,(I+\mu A)^{k-1}B]$.
\end{proof}

If $k>n$ then it is enough to consider powers of $I+\mu A$ only up to $n-1$.

\begin{proposition} \label{p23}
Let $\T=\mu \mathbb{Z}$, $k\geq n$, $t_0\in\T$ and $t_1=t_0+k\mu$. System \eqref{ukl1} is positively reachable on $[t_0,t_1]_\T$ iff the matrix
  $[B,(I+\mu A)B,\ldots,(I+\mu A)^{n-1}B]$
 contains a monomial submatrix.
\end{proposition}

For $\mu=1$ this was shown in \cite{FMT}. Then \eqref{ukl1} may be rewritten in a more familiar form $x(t+1)=(I+A)x(t)+Bu(t)$. The proof for $\mu\neq 1$ is very similar.

Proposition~\ref{p22} may be extended to nonhomogeneous discrete time scales.

\begin{proposition} \label{p24}
Assume that $\mu(t)>0$ for all $t\in\T$, $t_0\in\T$ and $t_1=\sigma^k(t_0)$. System \eqref{ukl1} is positively reachable on $[t_0,t_1]_\T$ iff the matrix
\begin{align*}
  [B,(I+\mu(\sigma(t_0)A)B,&(I+\mu(\sigma^2(t_0)A)(I+\mu(\sigma(t_0)A)B, \ldots,\\
  &(I+\mu(\sigma^{k-1}(t_0)A)\ldots(I+\mu(\sigma(t_0)A)B]
\end{align*} contains a monomial submatrix.
\end{proposition}

The proof is similar to the proof of Proposition~\ref{p22}, but we have to take into account that the exponential matrix is no longer a power of $I+\mu A$ for a constant $\mu$ but rather a product of such terms with possibly different values of $\mu$. This criterion may be used for systems on $\T=q^{\N}$.

\begin{rem}
\textup{Proposition~\ref{p23} cannot be extended to discrete nonhomogeneous time scales. Let us consider the following example: $\T=\{ 0,1,2,4\}$, $n=2$,
\[ A=\begin{pmatrix} -\frac{1}{2} & 0 \\
                           1      & -\frac{1}{2} \end{pmatrix},\
                           B=\begin{pmatrix} 1 \\
                                             0 \end{pmatrix}. \]
Let $t_0=0$ and $t_1=4$, so $k=3$. The matrix
\[  [B,(I+\mu(1)A)B,(I+\mu(2)A)(I+\mu(1)A)B]=\begin{pmatrix} 1 & \frac{1}{2} & 0 \\
                                                             0 &    1        & 3 \end{pmatrix} \]
evidently contains a monomial $2\times 2$ submatrix. But this is not true for the matrix $[B,(I+\mu(1)A)B]$. Thus to reach monomial vectors we may need more than $n$ jumps.}
\end{rem}


\begin{thebibliography}{xx}
\bibitem{BP2}  Z. Bartosiewicz and E. Paw{\l}uszewicz,  ``Unification of
continuous-time and discrete-time systems: the linear case'', in {Proceedings of Sixteenth International
Symposium on Mathematical Theory of Networks and Systems (MTNS2004)}, Katholieke Universiteit Leuven, Belgium,
2004.
\bibitem{BP3} { Z. Bartosiewicz and E. Paw{\l}uszewicz}, ``Realizations of linear control
systems on time scales'', {\emph{Control Cybernet.},} 35, pp.~769--786, 2006.
\bibitem{Bh} {M. Bohner and A. Peterson},  \emph{Dynamic Equations on Time Scales},
Birkh\"{a}user, Boston, 2001.
\bibitem{BRS} { R. Bru, S. Romero and E. Sanchez}, ``Canonical forms for positive discrete-time linear control systems'', \emph{Linear Algebra Appl.}, 310, pp.~49--71, 2000.
\bibitem{C} { Ch. Commault}, ``A simple graph theoretic characterization of reachability for positive linear systems'',  \emph{Systems \& Control Letters} 52, pp.~275-282, 2004.
\bibitem{CA} { Ch. Commault and M. Alamir}, ``On the reachability in any fixed time for positive continuous-time linear systems'', \emph{Syst. Control Lett.}, 56, pp.~272--276, 2007.
\bibitem{CS} { P.~G. Coxson and H. Shapiro}, ``Positive input reachability and controllability of positive linear systems'', \emph{Linear Algebra Appl.}, 94, pp.~35--53, 1987.
\bibitem{FMT} { M. Fanti, B. Maione and B. Turchiano}, ``Controllability of multiinput positive discrete-time systems'', \emph{Int. J. Control}, 51, pp.~1295--1308, 1990.
\bibitem{FR} { L. Farina and S. Rinaldi}, \emph{Positive Linear Systems:
Theory and Applications}, Pure and Applied Mathematics, John Wiley \& Sons, New York, 2000.
\bibitem{H} { S. Hilger}, ``Ein Ma\ss kettenkalk\"{u}lmit Anwendung auf
Zentrumsmannigfaltigkeiten'', Ph.D. thesis, Universit\"{a}t W\"{u}rzburg, 1988.
\bibitem{Ka} { T. Kaczorek}, \emph{Positive 1D and 2D Systems},
Springer-Verlag, London, 2002.
\bibitem{Ka1} { T. Kaczorek}, ``New reachability and observability tests for positive linear
discrete-time systems'', \emph{Bull. Pol. Acad. Sci., Technical Sciences}, 55, pp.~19--21, 2007.
\bibitem{OMK} { Y. Ohta, H. Maeda and S. Kodama}, ``Reachability, observability and realizability of continuous-time positive systems'', \emph{SIAM J. Control Optim.}, 22, pp.~171--180, 1984.
\bibitem{RJ} { V.~G. Rumchev and D.~J.~G. James}, ``Controllability of positive discrete-time systems'', \emph{Int. J. Control}, 50, pp.~845--857, 1989.
\bibitem{V} { M.~E. Valcher}, ``Reachability properties for continuous-time positive systems'', \emph{IEEE Trans. Automat. Control}, 54, pp.~1586--1590, 2009.
\end{thebibliography}
\end{document}